\documentclass[11pt]{amsart}
\usepackage{graphics}
\usepackage{caption}
\usepackage{geometry}
\usepackage{tikz}
\usepackage{dsfont}
\usepackage{amsmath}
\usepackage{amssymb}
\usepackage{amsthm}
\usepackage{vmargin}
\usepackage{hyperref}
\usepackage{lmodern}
\usepackage[cp1250]{inputenc}

\usepackage[shortlabels]{enumitem}

\newcommand{\RR}{\ensuremath{\mathbb{R}}}
\newcommand{\CC}{\ensuremath{\mathbb{C}}}

\newcommand{\dd}{\mbox{d}}
\newcommand{\NN}{\ensuremath{\mathbb{N}}}

\newcommand{\op}[1]{\operatorname{#1}}
\numberwithin{equation}{section}

\theoremstyle{plain}
\newtheorem{theor}{Theorem}[section]
\newtheorem{prop}[theor]{Proposition}
\newtheorem{lem}[theor]{Lemma}

\theoremstyle{definition}
\newtheorem{defin}[theor]{Definition}

\theoremstyle{remark}

\newtheorem*{rem}{Remark}

\title[Haagerup property of $q$-Araki-Woods algebras]{$q$-Araki-Woods algebras: extension of second quantisation and Haagerup approximation property}
\author{Mateusz Wasilewski}
\address{Institute of Mathematics of the Polish Academy of Sciences, ul. \'{S}niadeckich 8, 00-656 Warszawa, Poland}
\email{mwasilewski@impan.pl}
\thanks{The author was partially supported by the NCN (National Centre of Science) grant  
2014/14/E/ST1/00525 and acknowledges support by the French MAEDI and MENESR and by the
Polish MNiSW through the Polonium programme}

\date{}

\begin{document}
\begin{abstract}
We extend the class of contractions for which the second quantisation on $q$-Araki-Woods algebras can be defined. As a corollary, we prove that all $q$-Araki-Woods algebras possess the Haagerup approximation property.
\end{abstract}

\maketitle
\section{Introduction}
The Haagerup approximation property, along with amenability and weak amenability, started its life as an approximation property of (discrete) groups, although it was always intimately connected with operator algebras, beginning from its first appearance in \cite{MR520930}. This connection was further developed by Choda (cf. \cite{MR718798}), who defined the respective property for tracial von Neumann algebras and proved that a group von Neumann algebra (of a discrete group) possesses the Haagerup property if and only if so does the underlying group. The situation in the general locally compact case is, however, not that pleasant. It resembles the situation with amenability -- injectivty of group von Neumann algebra captures amenability of the group in the discrete case, but not in general. 

Ever since the advent of locally compact quantum groups and their approximation properties (cf. \cite{MR3456763}), it has become crucial to extend many notions beyond the case of finite von Neumann algebras. As in the classical case, there is no hope to define the Haagerup property of a general locally compact quantum group only via its von Neumann algebra. Nevertheless, in the discrete case this should be feasible. In \cite{MR3456763} the authors prove the proposed equivalence for \emph{unimodular} discrete quantum groups. The theory of quantum groups, however, has the unusual feature allowing discrete groups to be non-unimodular. This is a clear motivation to investigate the possibility of extending the definition to the case of non-tracial von Neumann algebras. Recently, two equivalent axiomatisations of the Haagerup property of general von Neumann algebras have been established (cf. \cite{MR3431616} and \cite{MR3395459}).

Whenever a new property is defined, it is useful to have a host of examples to confirm that the definition is a reasonable one. In this article we prove that a wide class of type III von Neumann algebras, the so-called $q$-Araki-Woods algebras introduced by Hiai in \cite{MR2018229} (based on earlier work of Shlyakhtenko, cf. \cite{MR1444786}), possess the Haagerup approximation property. It is a natural extension of the fact that the $q$-Gaussian algebras of Bożejko and Speicher (cf. \cite{MR1463036}) possess the Haagerup property, which seems to be a folklore result.

One can also view this paper as a contribution to the study of the structure of $q$-Araki-Woods algebras. A lot is known about their predecessors, the $q$-Gaussian algebras. They are known to be factors (cf. \cite{MR2164947}), they are non-injective (cf. \cite{MR2091676}), they possess the completely contractive approximation property (cf. \cite{1110.4918}). In the case of $q$-Araki-Woods algebras we have only partial results, e.g. a recent development in the study of factoriality (cf. \cite{1606.04752} and \cite{1607.04027}). The best known result about non-injectivity was obtained by Nou in \cite[Corollary 3]{MR2091676}). So far, the CCAP has been obtained only for free Araki-Woods algebras (\cite{MR2822210}); the general case, however, is likely to require new methods. We hope that this article will prompt further study of $q$-Araki-Woods algebras.

Let us give a brief overview of the paper. In Section \ref{prelim} we introduce the necessary definitions and tools. In Section \ref{extquant} we provide an extension of the second quantisation procedure, necessary for the proof of the Haagerup approximation property. The basic idea is that second quantisation allows us to build approximants on the level of the Hilbert space, which is easier than working directly on the level of the von Neumann algebra.  Section \ref{Haag} contains the main result of this article, which is the following theorem.
\begin{theor}\label{approx}
Let $(\mathcal{H}_{\RR}, (\mathsf{U}_t)_{t \in \RR})$ be a separable\footnote{We impose this condition so that the resulting von Neumann algebra acts on a separable Hilbert space.}, real Hilbert space equipped with a one parameter group of orthogonal transformations $(\mathsf{U}_t)_{t \in \RR}$. Then the $q$-Araki-Woods algebra $\Gamma_q(\mathcal{H}_{\RR}, (\mathsf{U}_t)_{t \in \RR})''$ has the Haagerup approximation property.
\end{theor} 
\section{Preliminaries}\label{prelim}
In this section we recall the construction of the $q$-Araki-Woods algebras and the definition of the Haagerup approximation property.
\subsection{q-Araki-Woods algebras}
The material about, $q$-Araki-Woods algebras that follows, with a much more detailed exposition, can be easily found in \cite{MR2018229}.

We start from a real separable Hilbert space $\mathcal{H}_{\RR}$ equipped with a one-parameter group of orthogonal transformations $(\mathsf{U}_t)_{t \in \RR}$. This extends to a unitary group on the complexification, denoted $\mathcal{H}_{\CC}$, that has the form $\mathsf{U}_t = A^{it}$ for some positive, injective operator $A$. We define a new inner product on $\mathcal{H}_{\CC}$ by $\langle x, y \rangle_{\mathsf{U}} := \langle \frac{2A}{1+A} x, y \rangle$. The completion with respect to this inner product is denoted by $\mathcal{H}$. Let us denote by $I$ the conjugation on $\mathcal{H}_{\CC}$ -- it is a closed operator on $\mathcal{H}$ because the new inner product coincides with the old one on $\mathcal{H}_{\RR}$. Consider now the $q$-Fock space $\mathcal{F}_q(\mathcal{H})$ (cf. \cite{MR1463036}).
\begin{defin}
For any $h \in \mathcal{H}_{\RR}$ define $s_q(h) = a_{q}^{\ast}(h) + a_q(h)$, where $a_q^{\ast}(h)$ and $a_q(h)$ are the creation and annihilation operators on $\mathcal{F}_q(\mathcal{H})$. \textbf{$q$-Araki-Woods algebra} is the von Neumann algebra generated by the set of operators $\{s_q(h): h \in \mathcal{H}_{\RR}\}$. We will denote it by $\Gamma_q(\mathcal{H})$. 

There are two special cases considered previously:
\begin{enumerate}
\item If the the group $(\mathsf{U}_t)_{t \in \RR}$ is trivial, i.e. $\mathsf{U}_t = \op{Id}$, then we denote the algebra $\Gamma_q(\mathcal{H})$ by $\Gamma_q(\mathcal{H}_{\RR})$ and call it a \textbf{$q$-Gaussian algebra} (cf. \cite{MR1463036});
\item If $q=0$, then $\Gamma_0(\mathcal{H})$ is called a \textbf{free Araki-Woods factor}; they were introduced earlier by Shlyakhtenko (cf. \cite{MR1444786})
\end{enumerate}

\end{defin}   
\begin{rem}
The \textbf{vacuum vector} $\Omega:= 1 \in \CC = \mathcal{H}^{\otimes 0} \subset \mathcal{F}_q(\mathcal{H})$ is both cyclic and separating for $\Gamma_q(\mathcal{H})$. Therefore, for $\xi$ in $\mathcal{F}_q(\mathcal{H})$, if there exists an operator $x \in \Gamma_q(\mathcal{H})$ such that $x\Omega=\xi$, then this operator is unique; we denote it by $W(\xi)$ and call it a \textbf{Wick word}.
\end{rem}
We would like to recall a useful construction of operators on $q$-Fock spaces.
\begin{defin}
Let $T: \mathcal{K} \to \mathcal{H}$ be a contraction between two Hilbert spaces. Then there exists a contraction $\mathcal{F}_q(T):\mathcal{F}_q(\mathcal{K}) \to \mathcal{F}_q(\mathcal{H})$, called \textbf{first quantisation} of $T$, which is defined on finite tensors by $\mathcal{F}_q(T)(v_1 \otimes \dots\otimes v_n) = Tv_1 \otimes \dots \otimes Tv_n$.
\end{defin}
\subsection{Haagerup approximation property}
We will follow the approach of Caspers and Skalski (cf. \cite{MR3431616}); for a different approach, based on standard forms, see \cite{MR3395459}. 
\begin{defin}
Let $(\mathsf{M}, \varphi)$ be a von Neumann algebra (with separabla predual) equipped with a normal, faithful, semifinite weight $\varphi$. It has \textbf{Haagerup approximation property} if there exists a sequence of unital, normal, completely positive (unital, completely positive will be abbreviated to ucp from now on) maps $(T_k: \mathsf{M} \to \mathsf{M})_{k \in \NN}$ such that:
\begin{enumerate}[{\normalfont (i)}]
\item $\varphi \circ T_k \leqslant \varphi$ for all $k \in \NN$;
\item GNS-implementations $T_k: L^{2}(\mathsf{M}, \varphi) \to L^{2}(\mathsf{M}, \varphi)$ are compact and converge to $\mathds{1}_{L^{2}(\mathsf{M}, \varphi)}$ strongly.
\end{enumerate}
\begin{rem}
It was proved in \cite{MR3431616} that Haagerup approximation property is an intrinsic property of the von Neumann algebra $\mathsf{M}$, i.e. it does not depend on the choice of the normal, faithful, semifinite weight.
\end{rem}
\end{defin}
\section{Second quantisation}\label{extquant}
In this section we will prove that second quantisation can be defined for arbitrary contractions on $\mathcal{H}_{\RR}$ that extend to contractions on $\mathcal{H}$; this condition will be written succinctly as $ITI = T$, where the left-hand side is understood as the closure\footnote{Note that, a priori, $ITI$ can be unbounded and not even densely defined.} of the product. Motivation comes from the paper \cite{MR2822210}, where the analogous generalisation of second quantisation is an indispensable tool for obtaining approximation properties in the free case. Before we give the details of the proof, let us first recall how to show that the second quantisation is always available in the case of $q$-Gaussian algebras so that the similarities and the differences are clearly visible (cf. \cite[Theorem 2.11]{MR1463036}). Before that, we need to recall the Wick formula (cf. \cite[Proposition 2.7]{MR1463036}). 
\begin{lem}\label{wick}
Suppose that $e_1,\dots, e_n \in \mathcal{H}_{\CC}$. Then
\begin{equation}\label{Wickformula}
W(e_1\otimes\dots\otimes e_n) = \sum_{k=0}^{n} \sum_{i_1,\dots,i_k,j_{k+1},\dots,j_n} a_q^{\ast}(e_{i_1})\dots a_q^{\ast}(e_{i_k})a_q(I e_{j_{k+1}})\dots a_q(I e_{j_n}) q^{i(I_1,I_2)},
\end{equation}
where $I_1 = \{i_1<\dots< i_k\}$ and $I_2 = \{j_{k+1}<\dots<j_n\}$ form a partition of the set $\{1,\dots,n\}$ and $i(I_1,I_2)$ is the number of crossings between $I_1$ and $I_2$, equal to $\sum_{l=1}^{k} (i_{l}-l)$. 
\end{lem}
\begin{theor}[{ \cite[Theorem 2.11]{MR1463036} }]
\label{bozejko}
Let $\mathcal{K}_{\RR}$ and $\mathcal{H}_{\RR}$ be real Hilbert spaces and let $T: \mathcal{K}_{\RR} \to \mathcal{H}_{\RR}$ be a contraction. Then there exists a ucp map $\Gamma_q(T): \Gamma_q(\mathcal{K}_{\RR}) \to \Gamma_q(\mathcal{H}_{\RR})$ such that $\Gamma_q(T) W(e_1\otimes \dots\otimes e_n) = W(Te_1\otimes\dots\otimes Te_n)$ for any $e_1,\dots,e_n \in \mathcal{K}_{\RR}$. Moreover, this map preserves the vacuum state.
\end{theor}
\begin{proof}
To prove the existence, we will first dilate $T$ to an orthogonal transformation $U_T$, i.e. define $U_T = \left[ \begin{array}{cc} \left(\mathds{1}_{\mathcal{K}_{\RR}} - T^{\ast}T\right)^{\frac{1}{2}}& T^{\ast} \\ T & -\left(\mathds{1}_{\mathcal{H}_{\RR}} - TT^{\ast}\right)^{\frac{1}{2}} \end{array} \right]$, an orthogonal operator on $\mathcal{K}_{\RR} \oplus \mathcal{H}_{\RR}$ such that $T = PU_T \iota$, where $\iota: \mathcal{K}_{\RR} \to \mathcal{K}_{\RR} \oplus \mathcal{H}_{\RR}$ is the inclusion onto the first summand and $P: \mathcal{K}_{\RR} \oplus \mathcal{H}_{\RR} \to \mathcal{H}_{\RR}$ is the orthogonal projection onto the second summand. We will define separately $\Gamma_{q}(\iota)$, $\Gamma(U_T)$, and $\Gamma_q(P)$ and then define $\Gamma_q(T):=\Gamma_q(P)\Gamma_q(U_T) \Gamma_q(\iota)$. The maps $\Gamma_q(P)$ and $\Gamma_q(U_T)$ are easy to define, so we will start with them. We define $\Gamma_q(P)x:= \mathcal{F}_q(P)x \mathcal{F}_q(P)^{\ast}$. This is a normal ucp map from $B(\mathcal{F}_q(\mathcal{K}_{\CC}\oplus\mathcal{H}_{\CC}))$ to $B(\mathcal{F}_q(\mathcal{H}_{\CC}))$, we just have to check that it maps $W(e_1\otimes\dots\otimes e_n)$ to $W(Pe_1\otimes\dots\otimes Pe_n)$. To this end, we will use the Wick formula \eqref{Wickformula}. It suffices to show that
\[
\Gamma_q(P)(a_q^{\ast}(v_1)\dots a_q^{\ast}(v_k)a_q(v_{k+1})\dots a_q(v_n)) = a_q^{\ast}(Pv_1)\dots a_q^{\ast}(Pv_k)a_q(Pv_{k+1})\dots a_q(Pv_n).
\]
We will use the fact that $a_q(v)\mathcal{F}_q(T) = \mathcal{F}_q(T)a_q(T^{\ast}v)$ and $\mathcal{F}_q(T) a_q^{\ast}(v) = a_q^{\ast}(Tv)\mathcal{F}_q(T)$. An easy application of this shows that $\mathcal{F}_q(P)a_q^{\ast}(v_1)\dots a_q^{\ast}(v_k)a_q(v_{k+1})\dots a_q(v_n)\mathcal{F}_q(P)^{\ast}$ is equal to $a_q^{\ast}(Pv_1)\dots a_q^{\ast}(Pv_k)\mathcal{F}_q(PP^{\ast})a_q(Pv_{k+1})\dots a_q(Pv_n)$ and we are done, because $PP^{\ast} = \mathds{1}_{\mathcal{H}_{\RR}}$. We define $\Gamma_q(U_T)$ analogously: $\Gamma_q(U_T)x  = \mathcal{F}_q(U_T) x \mathcal{F}_q(U_T)^{\ast}$. The same computation as in the case of $P$ shows that $\Gamma_q(U_T) W(e_1\otimes\dots\otimes e_n) = W(U_T e_1 \otimes \dots \otimes U_T e_n)$.

Now we have to deal with $\Gamma_q(\iota)$. Since $\iota \iota^{\ast} \neq \mathds{1}_{\mathcal{K}_{\RR} \oplus \mathcal{H}_{\RR}}$, the previous approach does not work. We know, however, that $\Gamma_q(\iota)$ ought to be the inclusion of $\Gamma_q(\mathcal{K}_{\RR})$ onto a von Neumann subalgebra of $\Gamma_q(\mathcal{K}_{\RR} \oplus \mathcal{H}_{\RR})$ generated by the operators $\{s_q(v): v \in \mathcal{K}_{\RR} \oplus \{0\} \subset \mathcal{K}_{\RR} \oplus \mathcal{H}_{\RR}\}$; denote the latter by $\Gamma_q(\mathcal{K}_{\RR}, \mathcal{K}_{\RR} \oplus \mathcal{H}_{\RR})$. To construct $\Gamma_q(\iota)$, we will define a map from $\Gamma_q(\mathcal{K}_{\RR}, \mathcal{K}_{\RR} \oplus \mathcal{H}_{\RR})$ onto $\Gamma_q(\mathcal{K}_{\RR})$ and show that it is an injective, hence isometric, $\ast$-homomorphism, therefore it has an inverse, which will be the sought $\Gamma_q(\iota)$. So far, we have a map $\Gamma_q(\iota^{\ast}): \Gamma_q(\mathcal{K}_{\RR} \oplus \mathcal{H}_{\RR}) \to \Gamma_q(\mathcal{K}_{\RR})$. Let us show that this map, when restricted to $\Gamma_q(\mathcal{K}_{\RR}, \mathcal{K}_{\RR} \oplus \mathcal{H}_{\RR})$, is a $\ast$-homomorphism. To show that, note that every member of the generating set of $\Gamma_q(\mathcal{K}_{\RR}, \mathcal{K}_{\RR} \oplus \mathcal{H}_{\RR})$ preserves the subspace $\mathcal{F}_q(\mathcal{K}_{\CC}) \subset \mathcal{F}_q(\mathcal{K}_{\CC} \oplus \mathcal{H}_{\CC})$; it follows that every element of $\Gamma_q(\mathcal{K}_{\RR}, \mathcal{K}_{\RR} \oplus \mathcal{H}_{\RR})$ enjoys this property. Let us take two elements $x,y \in \Gamma_q(\mathcal{K}_{\RR}, \mathcal{K}_{\RR} \oplus \mathcal{H}_{\RR})$ and compute
\[
\Gamma_q(\iota^{\ast})(xy) = \mathcal{F}_q(\iota^{\ast})xy \mathcal{F}_q(\iota) = \mathcal{F}_q(\iota^{\ast})x \mathcal{F}_q(\iota) \mathcal{F}_q(\iota^{\ast}) y \mathcal{F}_q(\iota),
\]
where the second equality follows from the fact that $\mathcal{F}_q(\iota \iota^{\ast})$ is the orthogonal projection from $\mathcal{F}_q(\mathcal{K}_{\CC} \oplus \mathcal{H}_{\CC})$ onto $\mathcal{F}_q(\mathcal{K}_{\CC})$ and the image of $y\mathcal{F}_q(\iota)$ is contained in $\mathcal{F}_q(\mathcal{K}_{\CC})$. Therefore $\Gamma_q(\iota^{\ast}): \Gamma_q(\mathcal{K}_{\RR}, \mathcal{K}_{\RR} \oplus \mathcal{H}_{\RR}) \to \Gamma_q(\mathcal{K}_{\RR})$ is a $\ast$-homomorphism. We will check now that it is injective. Suppose then that $\Gamma_q(\iota^{\ast})x=0$ for some $x \in \Gamma_q(\mathcal{K}_{\RR}, \mathcal{K}_{\RR} \oplus \mathcal{H}_{\RR})$. It follows that $\Gamma_q(\iota^{\ast})x\Omega = 0$. We have $\Gamma_q(\iota^{\ast})x\Omega = \mathcal{F}_q(\iota^{\ast})x\mathcal{F}_q(\iota)\Omega$ and $\mathcal{F}_q(\iota) \Omega= \Omega$, seen now as the vacuum vector in $\mathcal{F}_q(\mathcal{K}_{\CC} \oplus \mathcal{H}_{\CC})$. But we already know that $x\Omega \in \mathcal{F}_q(\mathcal{K}_{\CC}) \subset \mathcal{F}_q(\mathcal{K}_{\CC} \oplus \mathcal{H}_{\CC})$, so from $\mathcal{F}_q(\iota^{\ast})x\Omega=0$ we can deduce that $x\Omega=0$, therefore $x=0$ as $\Omega$ is a separating vector for $\Gamma_q(\mathcal{K}_{\RR} \oplus \mathcal{H}_{\RR})$. We proved that $\Gamma_q(\iota^{\ast}): \Gamma_q(\mathcal{K}_{\RR}, \mathcal{H}_{\RR} \oplus \mathcal{H}_{\RR}) \to \Gamma_q(\mathcal{K}_{\RR})$ is an isometric $\ast$-isomorphism, hence it has an inverse and we call this inverse $\Gamma_q(\iota)$; it is clear that $\Gamma_q(\iota)W(e_1\otimes\dots\otimes e_n) = W(\iota e_1\otimes \dots \otimes \iota e_n)$. It is easy to see that the vacuum state is preserved, so this finishes the proof.
\end{proof}  
The following extension, with almost the same proof, is due to Hiai (cf. \cite{MR2018229}):
\begin{prop}\label{Hiai}
Let $(\mathcal{K}_{\RR}, (\mathsf{U}_t)_{t \in \RR})$ and $(\mathcal{H}_{\RR}, (\mathsf{V}_t)_{t \in \RR})$ be two real Hilbert spaces equipped with one parameter groups of orthogonal transformations. Suppose that $T: \mathcal{K}_{\RR} \to \mathcal{H}_{\RR}$ is a contraction such that $T \mathsf{U}_t = \mathsf{V}_t T$ for all $t \in \RR$. Then there is a normal ucp map $\Gamma_q(T): \Gamma_q(\mathcal{K}) \to \Gamma_q(\mathcal{H})$ extending $W(e_1\otimes \dots \otimes e_n) \mapsto W(Te_1 \otimes \dots \otimes Te_n)$.
\end{prop}
\begin{proof}
We decompose $T=PU_T \iota$ as previously; the exact form of this decomposition is important. We equip the space $\mathcal{K}_{\RR} \oplus \mathcal{H}_{\RR}$ with the orthogonal group $(\mathsf{U}_t \oplus \mathsf{V}_t)_{t \in \RR}$. Note that the completion of $\mathcal{K}_{\CC} \oplus \mathcal{H}_{\CC}$ with respect to the inner product defined by $(\mathsf{U}_t \oplus \mathsf{V}_t)_{t \in \RR}$ is naturally identified with $\mathcal{K} \oplus \mathcal{H}$. Then the three maps $P$, $U_T$, and $\iota$ intertwine the orthogonal groups and, therefore, extend to contractions between appropriate Hilbert spaces. The rest of the proof is exactly the same as previously.
\end{proof}
We would like to state now our extension of the second quantisation (with the same minimal requirements as in \cite[Corollary 3.16]{MR2822210}).
\begin{theor}
Suppose that $T: \mathcal{K} \to \mathcal{H}$ is a contraction such that $T=JTI$, where $I$ is the conjugation on $\mathcal{K}_{\CC}$ and $J$ is the conjugation on $\mathcal{H}_{\CC}$. Then the assignment $W(e_1\otimes\dots\otimes e_n) \mapsto W(Te_1\otimes\dots\otimes Te_n)$ extends to a normal ucp map $\Gamma_q(T): \Gamma_q(\mathcal{K}) \to \Gamma_q(\mathcal{H})$ that preserves the vacuum state. 
\end{theor}
\begin{proof}
We start similarly as in the proof of Theorem \ref{bozejko}; dilate $T$ to a unitary $U_T$ on $\mathcal{K} \oplus \mathcal{H}$ given by $\left[ \begin{array}{cc} (\mathds{1}_{\mathcal{K}} - T^{\ast}T)^{\frac{1}{2}} & T^{\ast}  \\ T  & -(\mathds{1}_{\mathcal{H}} - TT^{\ast})^{\frac{1}{2}} \end{array} \right]$ so that $T = P U_T \iota$, where $\iota: \mathcal{K} \to \mathcal{K} \oplus \mathcal{H}$ is the natural inclusion and $P: \mathcal{K} \oplus \mathcal{H} \to \mathcal{H}$ is the orthogonal projection. Note that only $U_T$ depends on $T$, so it is easy to see that $\iota$ and $P$ come from maps of real Hilbert spaces $\mathcal{K}_{\RR}$, $\mathcal{K}_{\RR} \oplus \mathcal{H}_{\RR}$, and $\mathcal{H}_{\RR}$ and they intertwine the orthogonal groups $(\mathsf{U}_t)_{t \in \RR}$, $(\mathsf{U}_t \oplus \mathsf{V}_t)_{t \in \RR}$, and $(\mathsf{V}_t)_{t \in \RR}$. Therefore there is no problem with defining the second quantisation for these maps (Proposition \ref{Hiai}). We get a ucp map $\Gamma_q(\iota): \Gamma_q(\mathcal{K}) \to \Gamma_q(\mathcal{K} \oplus \mathcal{H})$. The condition $JTI=T$ is not self-adjoint, hence in general $U_T$ does not commute with $I\oplus J$, so there is no hope of defining a map $\Gamma_q(U_T): \Gamma_q(\mathcal{K} \oplus \mathcal{H}) \to \Gamma_q(\mathcal{K} \oplus \mathcal{H})$. However, there is a map $\mathcal{T}_q(U_T): \op{B}(\mathcal{F}_q(\mathcal{K} \oplus \mathcal{H})) \to \op{B}(\mathcal{F}_q(\mathcal{K} \oplus \mathcal{H}))$ given by conjugation $x \mapsto \mathcal{F}_q(U) x \mathcal{F}_q(U_T)^{\ast}$. One can easily check that $\mathcal{T}_q(U_T) \left(a_q^{\ast}(e_1)\dots a_q^{\ast}(e_k)a_q(e_{k+1})\dots a_q(e_n)\right) = a_q^{\ast}(U_T e_1)\dots a_q^{\ast}(U_T e_k)a_q(U_T e_{k+1})\dots a_q(U_T e_n)$. So far, we have a normal ucp map $\mathcal{T}_q(U_T)\circ \Gamma_q(\iota): \Gamma_q(\mathcal{K}) \to B(\mathcal{F}_q(\mathcal{K} \oplus \mathcal{H}))$. We now have to deal with the projection $P$. As in the case of a unitary operator, we get a map $\mathcal{T}_q(P): B(\mathcal{F}_q(\mathcal{K} \oplus \mathcal{H})) \to B(\mathcal{F}_q(\mathcal{H}))$ given by $x \mapsto \mathcal{F}_q(P)x \mathcal{F}_q(P)^{\ast}$. It is a simple matter to check that in this case we still have 
\[\mathcal{T}_q(P) \left(a_q^{\ast}(e_1)\dots a_q^{\ast}(e_k)a_q(e_{k+1})\dots a_q(e_n)\right) = a_q^{\ast}(Pe_1)\dots a_q^{\ast}(Pe_k) a_q(Pe_{k+1})\dots a_q(Pe_n).
\]
Finally, we obtain a (normal) ucp map 
$$
v:=\mathcal{T}_q(P)\circ \mathcal{T}_q(U_T) \circ \Gamma_q(\iota): \Gamma_q(\mathcal{K}) \to B(\mathcal{F}_q(\mathcal{H}))
$$
that has the property that $v(W(e_1\otimes\dots\otimes e_n))$ is equal to  
\[\sum_{k=0}^{n} \sum_{i_1,\dots,i_k,j_{k+1},\dots,j_n} a_q^{\ast}(PU_T\iota e_{i_1})\dots a_q^{\ast}(PU_T\iota e_{i_k})a_q(PU_T\iota I e_{j_{k+1}})\dots a_q(PU_T\iota I e_{j_n}) q^{i(I_1,I_2)}.
\]
Since $T=PU_T\iota$ satisfies $JTI=T$ this is equal to $W(Te_1\otimes\dots\otimes Te_n)$, therefore the image of $v$ is contained in $\Gamma_q(\mathcal{H})$ and we define $\Gamma_q(T):= v$.
\end{proof}
\subsection{Toeplitz algebra}
In this section we would like to present a second  approach to the extended second quantisation. Let us start with a definition.
\begin{defin}
Let $\mathcal{H}$ be a (complex) Hilbert space. Let $\mathcal{F}_q(\mathcal{H})$ be the $q$-Fock space over $\mathcal{H}$. We define the \textbf{$q$-Toeplitz algebra} $\mathcal{T}_q(\mathcal{H})$ to be the $C^{\ast}$-algebra generated by the creation operators $a_q^{\ast}(v)$ inside $B(\mathcal{F}_q(\mathcal{H}))$. If $\mathcal{K} \subset \mathcal{H}$ is a closed subspace, we define $\mathcal{T}_q(\mathcal{K}, \mathcal{H})$ to be the $C^{\ast}$-subalgebra of $\mathcal{T}_q(\mathcal{H})$ generated by the set $\{a_q^{\ast}(v): v \in \mathcal{K}\}$.
\end{defin}
\begin{rem}
Note that the algebra $\mathcal{T}_q(\mathcal{K}, \mathcal{H})$ leaves the subspace $\mathcal{F}_q(\mathcal{K}) \subset \mathcal{F}_{q}(\mathcal{H})$ globally invariant.
\end{rem}
We would like to note that both $\Gamma_q(P)$ and $\Gamma_q(U_T)$ (denoted then by $\mathcal{T}_q(P)$ and $\mathcal{T}_q(U_T)$) can be defined on the level of the algebra $\mathcal{T}_q(\mathcal{H})$ by the same formula. If we could do that also for $\Gamma_q(\iota)$, we would be able to obtain a second quantisation procedure on the level of the $q$-Toeplitz algebra. The reasons for seeking such a generalisation are twofold. First, it is interesting in its own right because better understanding of the structure of the $q$-Toeplitz algebra has potential applications to the study of radial multipliers (cf. \cite{MR2822210} for the free case). Second, it allows us to use the approach of Houdayer and Ricard (cf. \cite[Theorem 3.15 and Corollary 3.16]{MR2822210}) to extend the second quantisation for the $q$-Araki-Woods algebras. Let us point out what obstacle has to be overcome. To show that we can define $\mathcal{T}_q(\iota)$, we would like to show that the $\ast$-homomorphism $\mathcal{T}_q(\iota^{\ast}): \mathcal{T}_q(\mathcal{K}, \mathcal{K} \oplus \mathcal{H}) \to \mathcal{T}_q(\mathcal{K})$ is injective. This is the hard part, because now the vacuum vector $\Omega$ is not separating anymore. The kernel $\op{ker}(\mathcal{T}_q(\iota^{\ast}))$ is formed by elements vanishing on the subspace $\mathcal{F}_q(\mathcal{K}) \subset \mathcal{F}_q(\mathcal{K} \oplus \mathcal{H})$. We will now state the triviality of the kernel explicitly.
\begin{theor}\label{kernel}
Let $\mathcal{K}$ and $\mathcal{H}$ complex Hilbert spaces, with inclusion $\iota: \mathcal{K} \to \mathcal{H}$. Then the $\ast$-homomorphism $\mathcal{T}_q(\iota^{\ast}): \mathcal{T}_q(\mathcal{K}, \mathcal{H}) \to \mathcal{T}_q(\mathcal{K})$ is injective.
\end{theor}
To make the theorem look plausible, we would like to state a lemma saying that the linear span of the products of generators of the $q$-Toeplitz algebra, a dense $\ast$-subalgebra of it, does not contain any nontrivial element of the kernel -- this shows that there are no obvious candidates for the elements of the kernel. Before that, let us introduce some useful notation.
\begin{defin}
Let $\mathcal{H}$ be a complex Hilbert space and let $\overline{\mathcal{H}}$ be its complex conjugate space. We define maps $a_q^{\ast}: \mathcal{H}^{\otimes k} \to B(\mathcal{F}_q(\mathcal{H}))$ and $a_q: \overline{\mathcal{H}}^{\otimes k} \to B(\mathcal{F}_q(\mathcal{H}))$ (the tensor products are simply algebraic tensor products) to be the linear extensions of the maps given on simple tensors by $a_q^{\ast}(e_1\otimes\dots \otimes e_k) = a_q^{\ast}(e_1)\dots a_q^{\ast}(e_k)$ and $a_q(\overline{e}_1\otimes\dots\otimes \overline{e}_k)=a_q(e_1)\dots a_q(e_k)$. For any $\mathbf{v}_k\otimes \overline{\mathbf{w}}_{n-k} \in \mathcal{H}^{\otimes k} \otimes \overline{\mathcal{H}}^{\otimes (n-k)}$ we also define $A_{k,n}(\mathbf{v}_k \otimes \overline{\mathbf{w}}_{n-k}):= a_q^{\ast}(\mathbf{v}_k) a_q(\overline{\mathbf{w}}_{n-k})$. Let us also define the space 
\[
\mathsf{T}(\mathcal{H}):= \bigoplus_{n=0}^{\infty} \left(\bigoplus_{k=0}^{n} \mathcal{H}^{\otimes k} \otimes \overline{\mathcal{H}}^{\otimes n-k}\right),
\]
where the direct sums and tensor products are algebraic. Direct sum of all the operators $A_{k,n}$ will be denoted by $A: \mathsf{T}(\mathcal{H}) \to \mathcal{T}_q(\mathcal{H})$. Note that if $\iota:\mathcal{K} \to \mathcal{H}$ is an inclusion of Hilbert spaces, then $A$ can be equally well viewed as a map from $\mathsf{T}(\mathcal{K})$ to $\mathcal{T}_q(\mathcal{K}, \mathcal{H})$, for which we will use the same notation. 
\end{defin}
Using basic algebraic manipulations, we can obtain the following lemma, whose proof will be omitted.
\begin{lem}\label{L:finkernel}
Let $\iota: \mathcal{K} \to \mathcal{H}$ be an inclusion of Hilbert spaces. Then the mapping $A: \mathsf{T}(\mathcal{K}) \to \mathcal{T}_q(\mathcal{K}, \mathcal{H})$ is injective. As a consequence,
any $x$ in the range of $A$ is not in the kernel of the map $\mathcal{T}_q(\iota^{\ast}): \mathcal{T}_q(\mathcal{K}, \mathcal{H}) \to \mathcal{T}_q(\mathcal{K})$. 
\end{lem}

Before proving Theorem \ref{kernel}, we need just one more lemma, which we precede with introduction of convenient notation.

Elements of the form $a_q^{\ast}(\mathbf{v}_n) a_q(\overline{\mathbf{w}}_n)$, where $\mathbf{v}_n, \mathbf{w}_n \in \mathcal{H}^{\otimes n}$ will be called elements of \textbf{length $n$}, and their \emph{non-closed} linear span will be denoted by $\left(\mathcal{T}_q(\mathcal{H})\right)_{n}$. Note that, by Lemma \ref{L:finkernel}, the subspaces $\left(\mathcal{T}_q(\mathcal{H})\right)_{n}$ are linearly independent for different $n$, therefore the notion of length is well defined. We will also find it useful to specify the notation for the orthogonal projections $P_n: \mathcal{F}_q(\mathcal{H}) \to \mathcal{H}_q^{\otimes n}$, where $\mathcal{H}_q^{\otimes n}$ denotes the $n$-fold tensor power of $\mathcal{H}$ endowed with a $q$-deformed inner product. Let us also introduce the maps $R_{n+k,k}^{\ast}: \mathcal{H}_q^{\otimes (n+k)} \to \mathcal{H}_q^{\otimes n} \otimes \mathcal{H}_q^{\otimes k}$ (cf. \cite[Lemma 2]{MR2091676}) by their action on simple tensors:
$$
R_{n+k,k}^{\ast}(v_{1}\otimes\dots\otimes v_{n+k}) = \sum_{|I_1|=n, |I_2|=k}q^{i(I_1,I_2)} v_{i_{1}}\otimes\dots\otimes v_{i_n} \otimes v_{j_{n+1}}\otimes \dots \otimes v_{j_{n+k}}, 
$$
with the same notation as in Lemma \ref{wick}.

\begin{lem}\label{norm}
Suppose that $x \in \left(\mathcal{T}_q(\mathcal{H})\right)_{n}$. Then:
\begin{enumerate}[{\normalfont (i)}]
\item\label{one} $P_{n+k}x P_{n+k} = \op{Id}_{n,k}(P_n x P_n \otimes \op{Id}_k) R_{n+k,k}^{\ast}$, where $\op{Id}_{n,k}: \mathcal{H}_q^{\otimes n} \otimes \mathcal{H}_q^{\otimes k} \to \mathcal{H}_q^{\otimes n+k}$ is the extension of the identity map $\mathcal{H}^{\otimes n} \otimes \mathcal{H}^{\otimes k} \to \mathcal{H}^{\otimes n+k}$, defined on algebraic tensor products.
\item\label{two} $\|x\| \leqslant C(q) \|P_n x P_n\|$, where $C(q)$ is a positive constant depending only on $q$. Consequently, $\|x\| \simeq \|P_n x P_n\|$.
\end{enumerate}
\end{lem}
We defer the proof of the lemma for later, as we would like to first show how it helps in proving the main result.

\begin{proof}[Proof of Theorem \ref{kernel}] 
First, we would like to show that the task of proving triviality of the kernel can be reduced to a slightly easier one. To show that the kernel is trivial, it suffices to look at positive elements, since the kernel is an ideal, in particular a $C^{\ast}$-algebra, therefore it is spanned by positive elements. Suppose that $x$ is in the kernel and is positive. There is an action of the circle group (in our case it is the interval $[0,2\pi]$ with endpoints identified) on $\mathcal{T}_q(\mathcal{K}, \mathcal{H})$ given by  $t \stackrel{\alpha}{\mapsto} \mathcal{F}_q(e^{it}) x \mathcal{F}_q(e^{-it})$. This action leaves the kernel invariant, therefore the element $\mathbb{E}x:= \frac{1}{2\pi} \int_{0}^{2\pi} \mathcal{F}_q(e^{it})x \mathcal{F}_q(e^{-it}) \dd t$ is also in the kernel and is invariant by the action of the circle group defined above (this action is used by Pimsner in \cite{MR1426840} to show the universality of the usual Toeplitz algebra). It is a simple matter to check that the fixed point subalgebra is equal to the closed linear span of the elements of the form $a_q^{\ast}(\mathbf{v}_n) a_q(\overline{\mathbf{w}}_n)$, where $\mathbf{v}_n \in \mathcal{H}^{\otimes n}, \overline{\mathbf{w}}_n \in \overline{\mathcal{H}}^{\otimes n}$ and $n$ ranges over non-negative integers, and $\mathbb{E}$ is a faithful conditional expectation onto this fixed point subalgebra. So it suffices to show that there are no non-zero positive elements in this fixed point subalgebra that are in the kernel.

Suppose that $x$ is in the kernel and belongs to the subalgebra fixed by the circle action. We need to show that $x_{|\mathcal{H}^{\otimes n}} = 0$ for any $n \geqslant 0$; note that, as vector spaces, $\mathcal{H}_q^{\otimes n} = \mathcal{H}^{\otimes n}$, which follows from \cite[Lemma 1]{MR2091676}. We will prove the statement inductively. Fix a sequence $(x_k)_{k \in \NN}$ that approximates $x$ in norm and is contained in the \emph{non-closed} sum of the subspaces $\left(\mathcal{T}_q(\mathcal{H})\right)_n$. Therefore every $x_k$ admits a decomposition $x_k = \sum_{l=0}^{n_k} x_k^{(l)}$, where $x_k^{(l)} \in \left(\mathcal{T}_q(\mathcal{H})\right)_{l}$ and $n_k$ is the smallest number such that $x_k^{(l)}=0$ for $l>n_k$.

We would like to now state explicitly the statement we intend to prove by induction: For every $n \in \NN \cup \{0\}$ $x_{|\mathcal{H}^{\otimes n}}=0$ and $\lim_{k \to \infty} \|\sum_{l=0}^{n} x_k^{(l)}\|=0$. Let us start with $n=0$. Our inductive statement for $n=0$ means just that $P_0xP_0=0$ and $\lim_{k \to \infty} \|x_k^{(0)}\|=0$. The first part translates to $x\Omega=0$ and it follows from the fact that $\Omega \in \mathcal{K}$ and $x$ belongs to the kernel. For the other part, it follows from Lemma \ref{norm} that an element $y_l$ of length $l$ satisfies an inequality $\|y_l\| \leqslant C(q) \|P_ly_lP_l\|$, where $C(q)$ is a positive constant depending only on $q$. In our case we get $\|x_k^{(0)}\| \leqslant C(q) \|P_0 x_k^{(0)}P_0\|$. We know that $P_0xP_0=0$ and $P_0x_kP_0$ converges to $P_0xP_0$ in norm. However, $P_0x_k P_0 = P_0 x_k^{(0)}P_0$, since elements of length greater than $0$ annihilate the range of $P_0$.

Assume now that our statement has been proved for $m < n$ -- we would like to show that it is also true for $n$. Use the decomposition $\mathcal{H} = \mathcal{K} \oplus \mathcal{K}^{\perp}$ to write $\mathcal{H}^{\otimes n} = \mathcal{K}^{\otimes n} \oplus \mathcal{H}'$, where $\mathcal{H}'$ is a direct sum of tensor products of the spaces $\mathcal{K}$ and $\mathcal{K}^{\perp}$, where at least one factor is equal to $\mathcal{K}^{\perp}$. We would like to show that $x$ restricted to each of the tensor products vanishes. Since $x$ is in the kernel, we get it for $x_{|\mathcal{K}^{\otimes n}}$. Let $\mathcal{K}'$ be any other summand. We will show that $x_{k}(\mathbf{e})$ converges to $0$ for any simple tensor $\mathbf{e} \in \mathcal{K}'$. Since $\mathbf{e}$ is of length $n$, we get $x_k(\mathbf{e}) = \sum_{l=0}^{n} x_k^{(l)}(\mathbf{e})$. By the inductive assumption, we know that $\sum_{l=0}^{n-1} x_k^{(l)}$ converges in norm to $0$, so we are left with $x_k^{(n)}(\mathbf{e})$. But every summand in $x_k^{(n)}$ is of the form $a_q^{\ast}(\mathbf{v}_n) a_q(\overline{\mathbf{w}}_n)$ and $a_q^{\ast}(\mathbf{v}_n) a_q(\overline{\mathbf{w}}_n) \mathbf{e} = \langle \Sigma \mathbf{w}_n, \mathbf{e}\rangle \mathbf{v}_n$, where $\Sigma: \mathcal{H}^{\otimes n} \to \mathcal{H}^{\otimes n}$ is the \textbf{flip map}, taking $h_1 \otimes \dots \otimes h_n$ to $h_n \otimes \dots \otimes h_1$. Since $\mathbf{e}$ possesses a vector from $\mathcal{K}^{\perp}$ in its tensor decomposition, $\langle \Sigma \mathbf{w}_n, \mathbf{e} \rangle = 0$. It follows that $x_{|\mathcal{H}^{\otimes n}}=0$. This implies that  $\lim_{k \to \infty} \|P_n x_k P_n\|=0$. Since $P_n x_k P_n = \sum_{l=0}^{n} P_n x_k^{(l)} P_n$ and $\lim_{k\to \infty} \|\sum_{l=0}^{n-1} x_k^{(l)}\|=0$, we get that $\lim_{k \to \infty} \|P_n x_k^{(n)} P_n\|=0$. Using Lemma \ref{norm}, we conclude that $\lim_{k\to \infty} \|x_k^{(n)}\|=0$, therefore $\lim_{k \to \infty} \|\sum_{l=0}^{n} x_k^{(l)}\|=0$.
\end{proof}
To finish the proof, we need to prove Lemma \ref{norm}. To do that, we will need one simple lemma.
\begin{lem}\label{L:majorisation}
Let $\mathcal{H}$ be a Hilbert space. Suppose that $A,B$ are positive operators on $\mathcal{H}$ and $T$ is a bounded operator such that $A=BT$. Then $A \leqslant \|T\| B$.
\end{lem}
\begin{proof}
By taking the adjoint, we get $A = T^{\ast}B$, hence $A^2 = BTT^{\ast}B$. It follows that $A^2 \leqslant \|T\|^2 B^2$. The majorisation $A \leqslant \|T\| B$ is implied by the operator monotonicity of the square root.
\end{proof}
\begin{proof}[Proof of Lemma \ref{norm}]
\begin{enumerate}[{\normalfont (i)}, wide, labelwidth=!, labelindent=0pt]
\item Fix $x$ of length $n$ -- it is a linear combination of elements of the form $a_q^{\ast}(\mathbf{v}_n) a(\overline{\mathbf{w}}_n)$. Since the formula
\begin{equation}\label{toeplitz}
P_{n+k} x P_{n+k} = (P_n x P_n \otimes \op{Id}_k) R_{n+k,k}^{\ast},
\end{equation}
is linear in $x$, it suffices to prove it for $x$ of the form $a_q^{\ast}(\mathbf{v}_n) a(\overline{\mathbf{w}}_n)$, where $\mathbf{v}_n=v_1\otimes\dots\otimes v_n$ and $\mathbf{w}_n=w_n\otimes \dots \otimes w_1$ are simple tensors.
Fix $\mathbf{e} \in \mathcal{H}^{\otimes n+k}$; we have to check that $x e = (x \otimes \op{Id}_k) R_{n+k,k}^{\ast} e$. Note that the action of creation operators does not depend on the tensor power on which they act -- it always boils down to tensoring by a vector on the left. Therefore we need only to concern ourselves with annihilation operators. We would like to express $q$-annihilation operators $a_q(v)$ in terms of free annihilation operators $a(v):=a_0(v)$. Note that $a_q^{\ast}(v) = a^{\ast}(v)$, so for any finite tensors $\mathbf{y}$ and $\mathbf{z}$ we get $\langle \mathbf{z}, a_q(v) \mathbf{y}\rangle_q = \langle a^{\ast}(v) \mathbf{z}, \mathbf{y} \rangle_q$. Let $P_q^n$ be the positive operator on $\mathcal{H}^{\otimes n}$ defining the $q$-deformed inner product; let us denote by $P_q$ the direct sum of all the operators $P_q^n$. Using the definition of the $q$-deformed inner product, we arrive at
$$
\langle \mathbf{z}, P_q a_q(v)\mathbf{y}\rangle_0 = \langle P_q a^{\ast}(v) \mathbf{z}, \mathbf{y} \rangle_0.
$$
It follows that $P_q a_q(v) = (P_q a^{\ast}(v))^{\ast} = a(v) P_q$, so $a_q(v) = P_q^{-1} a(v) P_q$. If we restrict this equality to $\mathcal{H}^{\otimes n}$, we get $a_q(v)_{|\mathcal{H}^{\otimes n}} = (P_q^{n-1})^{-1} a(v) P_q^n$. Let us compute the left-hand side of \eqref{toeplitz}:
$$
a_q(w_n)\dots a_q(w_1) \mathbf{e} = (P_q^k)^{-1}a(w_n)\dots a(w_1) P_q^{n+k} \mathbf{e}. 
$$
This formula follows from the fact that first we change $a_q(w_1)$ to $(P_q^{n+k-1})^{-1}a(w_1)P_q^{n+k}$, but then $a_q(w_2)$ has to be changed to $(P_q^{n+k-2})^{-1}a(w_2)P_q^{n+k-1}$ and there is a cancellation between $a(w_2)$ and $a(w_1)$; using this fact repeatedly, we obtain the above formula. To calculate the right-hand side, recall (cf. \cite[Formula $2$ on page $21$]{MR2091676}) that we have an equality $P_q^{n+k} = (P_q^n \otimes P_q^k)R_{n+k,k}^{\ast}$, so $R_{n+k,k}^{\ast} = ((P_q^n)^{-1} \otimes (P_q^k)^{-1}) P_q^{n+k}$. It leads us to: 
\begin{align}
&(a_q(w_n)\dots a_q(w_1) \otimes \op{Id}_k) ((P_q^n)^{-1} \otimes (P_q^k)^{-1}) P_q^{n+k}\mathbf{e} \notag \\
&= (a_q(w_n)\dots a_q(w_1) (P_q^n)^{-1} \otimes (P_q^k)^{-1})P_q^{n+k} \mathbf{e} \notag \\
&= (a(w_n)\dots a(w_1) \otimes (P_q^k)^{-1}) P_q^{n+k} \mathbf{e}. \notag
\end{align}
We now only need to understand that this is exactly the same formula. It follows from the fact that the free annihilation operators act only on the $n$ leftmost vectors, so the operator $(P_q^k)^{-1}$ in both situations acts only on the $k$ rightmost ones.
\item
First of all, since the spaces $\mathcal{H}_q^{\otimes k}$ are left invariant by $x$, we have $\|x\| = \sup_{k \geqslant 0} \|P_k x P_k\|$. Because $P_k x P_k = 0$ for $k<n$, we actually get $\|x\| = \sup_{k \geqslant 0} \|P_{n+k} x P_{n+k}\|$. We just have to show that $\|P_{n+k}x P_{n+k}\| \leqslant C(q) \|P_n x P_n\|$. From the first part of the proof we get that $\|P_{n+k}x P_{n+k}\| \leqslant \|\op{Id}_{n,k}\|\cdot\|R^{\ast}_{n+k,n}\|\cdot \|P_n x P_n\|$. It is known (cf. \cite[Formula 2 on page 21]{MR2091676}) that $\|R^{\ast}_{n+k,n}\| \leqslant C(q)$, where $C(q) = \prod_{k=1}^{\infty} (1-|q|^{k})^{-1}$ and $R^{\ast}_{n+k,n}$ is seen as an operator on $\mathcal{H}^{\otimes (n+k)}$, where $\mathcal{H}^{\otimes (n+k)}$ is equipped with the standard inner product, not the $q$-deformed one. It follows from Lemma \ref{L:majorisation} that $P_q^{(n+k)} \leqslant C(q) P_q^{(n)} \otimes P_q^{(k)}$ as operators on $\mathcal{H}^{\otimes (n+k)}$, since $P_q^{(n+k)} = (P_q^{(n)} \otimes P_q^{(k)})R^{\ast}_{n+k,k}$. Because $P_q^{(n+k)}$ defines the inner product on $\mathcal{H}_q^{\otimes (n+k)}$, and $P_q^{(n)} \otimes P_q^{(k)}$ defines the inner product on $\mathcal{H}_q^{\otimes n} \otimes \mathcal{H}_q^{\otimes k}$, it follows that the identity map $\op{Id}_{n,k}: \mathcal{H}_q^{\otimes n} \otimes \mathcal{H}_q^{\otimes k} \to \mathcal{H}_q^{\otimes (n+k)}$ has norm not greater than $\sqrt{C(q)}$. Finally, $(\op{Id}_{n,k})^{\ast} = R^{\ast}_{n+k,n}$, so $\|R^{\ast}_{n+k,n}\| \leqslant \sqrt{C(q)}$ as an operator mapping $\mathcal{H}_q^{\otimes (n+k)}$ to $\mathcal{H}_q^{\otimes n} \otimes \mathcal{H}_q^{\otimes k}$. This shows that $\|P_{n+k} x P_{n+k}\| \leqslant C(q) \|P_n x P_n\|$.
\end{enumerate}
\end{proof}
\section{Haagerup approximation property}\label{Haag}
To prove the Haagerup property, we need to use one more lemma.
\begin{lem}[Houdayer--Ricard, \cite{MR2822210}]
There exists a sequence $(T_k)_{k \in \NN}$ of finite rank contractions on $\mathcal{H}$ such that $IT_k I=T_k$ and $\lim_{k \to \infty} T_k = \mathds{1}$ strongly.
\end{lem}
We are now ready to prove our main result.
\begin{proof}[Proof of Theorem \ref{approx}]
Consider the ucp maps $v_{k,t} := \Gamma_q(e^{-t} T_k)$ -- they preserve the vacuum state. We would like to prove that the GNS-implementations of these maps converge strongly to identity and are compact. First of all, by definition, the GNS implementations of these maps are equal to $\mathcal{F}_q(e^{-t} T_k)$. Let us then check compactness. Recall that we denote by $P_n: \mathcal{F}_q(\mathcal{H}) \to \mathcal{F}_q(\mathcal{H})$ the orthogonal projection onto first $n$ summands in the direct sum decomposition of the Fock space. Since $T_k$ is a finite-rank operator, so is $P_n\mathcal{F}_q(e^{-t}T_k)$. We have to show that the norm of $P_n^{\perp}\mathcal{F}_q(e^{-t} T_k)$ converges to $0$, when $n \to \infty$. First of all, let us reduce to the case $q=0$. Operator $P_q$ preserves all the tensor powers appearing in the direct sum decomposition of the Fock space, therefore it commutes with $P_n^{\perp}$. It also commutes with the first quantisation operators $\mathcal{F}_q(e^{-t} T_k)$. It follows from Lemma $1.4$ in \cite{MR1463036} that the norm of $P_n^{\perp} \mathcal{F}_q(e^{-t} T_k)$ does not change if we compute it on the free Fock space $\mathcal{F}_0(\mathcal{H})$; this is the norm that we will estimate. This is easy when $T_k$ is self-adjoint and we will now show that one can assume that. Indeed, the first quantisation on the level of the Fock space interacts nicely with taking the adjoint, so we get (by the $C^{\ast}$-identity) 
$$
\|P_n^{\perp}\mathcal{F}_0(e^{-t}T_k)\|^2 = \|P_n^{\perp}\mathcal{F}_0 (e^{-2t}T_k^{\ast}T_k) P_n^{\perp}\|.
$$
Now $T_k^{\ast}T_k$ is a finite rank positive contraction, so there is an orthonormal basis $(e_i)_{i \in \NN}$ of $\mathcal{H}$ such that $T_k^{\ast}T_k e_i = \lambda_i e_i$ and $\lambda_i \in [0,1]$. From the orthonormal basis of eigenvectors of $T_k^{\ast}T_k$ we can build an orthonormal basis of $\mathcal{F}_0(\mathcal{H})$, using tensor powers; for a multi-index $I=\{i_1,\dots, i_k\}$ we will denote $e_{I} = e_{i_1} \otimes \dots \otimes e_{i_k}$ and $\lambda_I = \lambda_{i_1}\cdots\lambda_{i_k}$. We can now estimate the norm of $P_n^{\perp}\mathcal{F}_q (e^{-2t}T_k^{\ast}T_k) P_n^{\perp}$. Let $v \in \mathcal{F}_0(\mathcal{H})$ be written as $v = \sum a_I e_I$, then
\begin{align}
\|P_n^{\perp}\mathcal{F}_0 (e^{-2t}T_k^{\ast}T_k) P_n^{\perp} v\|^2 &= \|\sum_{|I|> n} e^{-2|I|} a_I \lambda_{I} e_{I}\|^2 \notag \\
&= \sum_{|I|> n} e^{-4|I|} |a_{I}|^2 |\lambda_{I}|^2 \notag \\
&\leqslant e^{-4(n+1)} \|v\|^2, \notag
\end{align}
because $|\lambda_I| \leqslant 1$. The fact that the operators $\mathcal{F}_q(e^{-t} T_k)$ converge strongly to the identity when $t \to 0$ and $k\to \infty$ is clear; it can be easily checked on finite simple tensors and this suffices, since they are all contractive. This ends the proof.
\end{proof}
\section{Acknowledgements}
I am grateful to my advisor, Adam Skalski, for helpful discussions and careful reading of the draft version of this paper. I would also like to thank Michael Brannan for useful remarks. 
\bibliographystyle{alpha}
\bibliography{../research}

\end{document}